\documentclass[11pt,makeidx]{amsart}
\usepackage[foot]{amsaddr}
\usepackage{enumerate}
\usepackage{amssymb,mathrsfs, amsmath, amsthm}
\usepackage{amsthm,amsmath,amsfonts,amssymb,mathtools,times}
\usepackage{dirtytalk}
\usepackage{soul}
\usepackage[T1]{fontenc}
\usepackage{graphicx}
\usepackage{upgreek}
\usepackage{cmap}

\graphicspath{ {./images/} }
\usepackage[usenames,dvipsnames]{xcolor}
\definecolor{darkblue}{rgb}{0.0, 0.0, 0.55}
\usepackage[
colorlinks,
linkcolor=BrickRed,citecolor=OliveGreen,urlcolor=darkblue,
hypertexnames=true]{hyperref}
\usepackage{tikz}
\usepackage{float}
\usepackage{caption}
\usepackage[capitalise, noabbrev]{cleveref}

\newtheorem{theorem}{Theorem}[section]
\newtheorem{lemma}[theorem]{Lemma}
\newtheorem{corollary}[theorem]{Corollary} 
\newtheorem{proposition}[theorem]{Proposition}
\newtheorem*{question}{Question}

\theoremstyle{definition}

\newtheorem{example}[theorem]{Example}

\theoremstyle{remark}

\DeclareMathOperator{\depth}{depth}
\DeclareMathOperator{\reg}{reg}
\DeclareMathOperator{\degh}{deg h}
\DeclareMathOperator{\Graph}{Graph}
\DeclareMathOperator{\CW}{CW}
\DeclareMathOperator{\G}{G}
\DeclareMathOperator{\Grdd}{\Graph_{\depth,\dim}}
\DeclareMathOperator{\CWdd}{\CW_{\depth,\dim}}
\DeclareMathOperator{\Ra}{\CW_{\depth,\reg,\dim,\degh}}
\DeclareMathOperator{\Raa}{\CW_{2,\reg,\dim,\degh}}
\DeclareMathOperator{\rd}{\CW_{RD}}
\DeclareMathOperator{\grd}{\G_{RD}}
\DeclareMathOperator{\RD}{CW_{\reg,\degh}}
\DeclareMathOperator{\m}{m}
\DeclareMathOperator{\im}{im}

\newcommand{\qor}{\quad \mbox{or} \quad}
\newcommand{\qand}{\quad \mbox{and} \quad}

\newcommand{\lf}{\left \lfloor}
\newcommand{\rf}{\right \rfloor}
\newcommand{\lc}{\left \lceil}
\newcommand{\rc}{\right \rceil}
\makeindex

\title[Counting Lattice Points of Cameron--Walker Graphs]{Counting Lattice Points that appear as algebraic invariants of Cameron--Walker Graphs}

\author[Sara Faridi, Iresha Madduwe Hewalage]{Sara Faridi, Iresha Madduwe Hewalage}
\address{Department of Mathematics\\
  Dalhousie University, Halifax \\ Canada}

\begin{document}
\today

\begin{abstract} 
In 2021, Hibi et. al. studied lattice points in $\mathbb{N}^2$
that appear as $(\depth R/I,\dim R/I)$ when $I$ is the edge ideal of a
graph on $n$ vertices, and showed these points lie between two convex
polytopes. When restricting to the class of Cameron--Walker graphs,
they showed that these pairs do not form a convex lattice polytope. In
this paper, for the edge ideal $I$ of a Cameron--Walker graph on $n$
vertices, we find how many points in $\mathbb{N}^2$ appear as
$(\depth(R/I),\dim(R/I))$, and how many points in $\mathbb{N}^4$
appear as $$(\depth(R/I),\reg(R/I),\dim(R/I),\degh(R/I)).$$

   \end{abstract}

\maketitle

\section{ Introduction}

Combinatorial commutative algebra studies the problems in commutative
algebra using the tools and techniques in combinatorics of geometric
structures.  Monomial ideals play an important role in studying the
relationship between commutative algebra and combinatorics.  Let $R=
K[x_1,\ldots, x_n]$ be the polynomial ring in $n$ variables over a
field $K$. For a finite simple graph $G$ on vertex set $[n] = \{1,
\ldots, n\}$ and the edge set $E(G)$ we define the edge ideal of $G$,
denoted by $I(G)$, as the monomial ideal of $R$, generated by those
monomials $x_ix_j$ with $\{i, j\} \in E(G)$. Edge ideals of finite
simple graphs have been studied by many authors from a viewpoint of
commutative algebra. In the present paper, we are interested in the
homological invariants 
\begin{eqnarray*}
\depth(G)= \depth(R/I(G)),&
\dim(G)=\dim(R/I(G)),\\
\reg(G)=\reg(R/I(G)), &
\degh(G)=\degh(R/I(G)).
\end{eqnarray*}

Katzman \cite{Kt} found that these homological invariants
do not depend on the characteristic of the field $K$, if our graph has
at most 10 vertices. Hibi, Kanno, Kimura, Matsuda and Van Tuyl, have
shown \cite{h, h1}, that when we restricted to the family of
Cameron--Walker graphs we can represent these homological invariants
in terms of the combinatorics of the graphs only. i.e. these
homological invariants are characteristic--free on the class of
Cameron--Walker graphs.  Cameron--Walker graphs are named after the
2005 paper \cite{k} of Cameron and Walker in which the authors
characterized the class of graphs with the property that the maximum
matching number equals the maximum induced matching number. As a
result of their work, they classified these graphs into 3 types; star
graphs, star triangles and finite graphs consisting of a connected
bipartite graph with two vertex partition sets such that there is at
least one leaf edge attached to each vertex in one vertex set and that
there may be possibly some triangles attached to the vertices in the
other vertex set.  Hibi, Higashitani, Kimura, and O’Keefe (\cite{Hi}),
named the graphs of the last type Cameron--Walker (CW)
graphs.  The relationships among the homological invariants of the
edge ideals of Cameron--Walker graphs have been studied by many
authors. In \cite{h1} authors studied the possible tuples
$(\reg(G),\degh(G))$ for all graphs $G$ on a fixed number of vertices.
They found that such lattice points lie between two convex lattice
polytopes. Moreover, it was shown that if we restricted to the class
of Cameron--Walker graphs the set of all such lattice points form a
convex lattice polytope. Inspired by this result, the authors in paper
\cite{h}, studied the ordered pairs

 \begin{equation}
 \label{1a}
    (\depth(G),\dim(G))
\end{equation}

for all the finite simple graphs $G$ with a fixed number of vertices
and they showed the set of all such points always lie between two 
convex lattice polytopes. However, when they restricted to the family of Cameron--Walker graphs the set of all possible lattice points of the form in \eqref{1a} do not form a convex
lattice polytope. Moreover, they completely determined all the possible pairs of the form \eqref{1a} arising from Cameron--Walker graphs $G$. Moving from this point in paper \cite[Theorem 4.4]{h} they found all the tuples 
 \begin{equation}
 \label{1b} (\depth(G),\reg(G),\dim(G),\degh(G))
  \end{equation}
of the edge ideals of Cameron--Walker graphs $G$. 

For Cameron--Walker graphs on $n$ vertices we define
$$\CWdd(n)=\{( \depth(G),\dim(G)) 
\mid G \mbox{ Cameron--Walker graph} \}$$ 
and 
$$\Ra(n)=\{(\depth(G),\reg(G),\dim(G), \degh(G))  
\mid G \mbox{ graph}\}.$$
         
Our goal is to find the size of the two sets $\CWdd(n)$ and $\Ra(n)$ without finding all the possible tuples for each set similar to the work done in \cite{h1} for the pair, regularity, and the degree of the h--polynomial of Cameron--Walker graphs of order $n\geq 5$..

We answer this question in this paper; we compute the number of
elements $(a,b)$ where $a= \depth(G)$ and $b=\dim(G)$ for the set
$\CWdd(n)$.  Next we try to compare the number of integer points in
$\CWdd(n)$ to the number of possible pairs $(a,b)$ in $\Grdd(n)$. To do this
we will find $${\lim_{n \rightarrow \infty} \frac{|\CWdd(n)|}{|\Grdd(n) |}}$$ assuming it exists. Finally, we determine the size
of the set of tuples $$ (\depth(G),\reg(G),\dim(G),\degh(G))$$ arising
from Cameron--Walker graphs on fixed number of vertices.

\section{Background}

A {\bf graph} is a mathematical structure consisting of a pair $G =
(V, E)$, where $V$ is a set of points called {\bf vertices}, and
$E$ is a collection of unordered pairs of vertices, whose elements are
called {\bf edges}.  A {\bf loop} is an edge that connects a
vertex to itself.  If a graph has more than one edge joining some pair
of vertices then these edges are called {\bf multiple edges}.  A
{\bf simple graph} is a graph without loops or multiple edges. All
graphs in this paper are simple graphs. For a simple graph, two
vertices are said to be {\bf adjacent} if there is an edge
connecting them. A {\bf path} is a sequence of distinct vertices
with the property that the consecutive vertices are adjacent,
i.e. each vertex $x_i$ in the sequence is adjacent by an edge to the
vertex $x_{i+1}$ next to it.  A path on $n$ vertices is denoted by
$P_n$. A graph is said to be {\bf connected} if any two of its
vertices are joined by a path.  In a simple graph, the {\bf degree} of a
vertex $v$, denoted $\deg(v)$, is the number of edges meeting it, or
in other words, incident to it.

 A {\bf complete graph} is a graph in which each pair of distinct 
 vertices is connected by an edge.  A complete graph on $n$ vertices
 is denoted by $K_n$. The complete graph $K_3$ is called a triangle.
 A {\bf bipartite graph} is a graph whose vertices can be divided
 into two sets, $U$ and $W$, such that every edge connects a vertex in
 $U$ to one in $W$. i.e.  there is no edge that connects vertices of
 same set.  A {\bf complete bipartite graph} is a bipartite graph
 such that every vertex in one set is adjacent to every vertex of
 other set. If the two sets have $m$ and $n$ vertices then we denote
 the complete bipartite graph by $K_{m,n}$.
 
Let $G=(V,E)$ be a graph. A {\bf matching} is a set of edges in $E$
no two of which share a common vertex.  An {\bf induced matching}
$M$ is a matching such that no pair of edges of $M$ are joined by an
edge in the graph $G$. 
 A matching $M$ is said to be {\bf maximum} if for any other
matching $M^{\prime},\  |M|\geq |M^{\prime}|$.
Here $|M|$ is the size of the matching $M$.
The maximum size of a matching in $G$ is denoted by $\m(G)$, and the maximum size of an induced matching in $G $ is denoted by $\im(G)$.

\begin{example}
 The edge set $\{af,bc,ed\}$ in \cref{G} represents a 
 maximum matching and  $\{af,cd\}$ gives a 
 maximum induced matching.
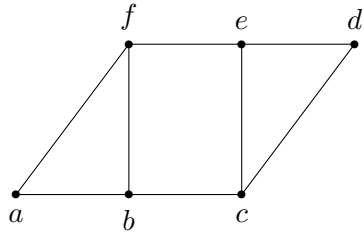
\begin{figure}[ht]
    $$\begin{tikzpicture}
\tikzstyle{point}=[inner sep=0pt]
\node (a)[point,label=below:$a$] at (-1.5,0) {\tiny{$\bullet$}};
\node (b)[point,label=below:$b$] at (0,0) {\tiny{$\bullet$}};
\node (c)[point,label=below:$c$] at (1.5,0) {\tiny{$\bullet$}};
\node (f)[point,label=above:$f$] at (0,2) {\tiny{$\bullet$}};
\node (e)[point,label=above:$e$] at (1.5,2) {\tiny{$\bullet$}};
\node (d)[point,label=above:$d$] at (3,2) {\tiny{$\bullet$}};
\draw (a.center) -- (b.center);
\draw (b.center) -- (c.center);
\draw (d.center) -- (e.center);
\draw (f.center) -- (e.center);
\draw (b.center) -- (f.center);
\draw (c.center) -- (e.center);
\draw (c.center) -- (d.center);
\draw (a.center) -- (f.center);

\end{tikzpicture}
$$
\caption{ A simple graph}
    \label{G}
\end{figure}

 \end{example}

  A {\bf leaf} in a graph is a
       vertex of degree 1, that is, a vertex which meets only one edge. A
       {\bf leaf edge} or a {\bf pendant edge} 
       is an edge which meets a leaf.
A {\bf pendant triangle} in a graph is a triangle in which exactly two vertices have degree 2 and the third vertex is of degree greater than 2. The pendant triangle is  attached to the graph at the vertex with degree greater than 2.
         
Cameron and Walker~\cite{k}  proved the following statement:

\begin{theorem}\cite[Theorem 1]{k}
\label{T1}
Let $G$ be a connected graph. Then $\m(G)=\im(G)$ 
  if and only if $G$ is a star or a star triangle, or consists
of a connected bipartite graph $B$ with vertex partition
$\{u_1,\dots,u_m\} \cup \{v_1,\dots,v_n\}$ such that there is at
least one leaf edge 
(and possibly more) attached to each vertex $u_i$
of $B$, and possibly some pendant triangles 
attached to a vertex $v_j$ of $B$.
\end{theorem}

In other words, a finite connected simple graph $G$ satisfies $\m(G)=\im(G)$ if and only if $G$ is one of the following types of graphs:

\begin{enumerate}
    \item a star , i.e. a complete bipartite graph $K_{1,n}$ as in
      \cref{f:star};
    
      \begin{figure}[ht]
        \begin{center}
          \begin{tabular}{ccc}
            \begin{tikzpicture}
\tikzstyle{point}=[inner sep=0pt]
\node (a)[point,label=below:$a$] at (0,0) {\tiny{$\bullet$}};
\node (b)[point,label=below:$b$] at (-3/2,3/4) {\tiny{$\bullet$}};
\node (g)[point,label=right:$g$] at (0,3/2) {\tiny{$\bullet$}};
\node (c)[point,label=below:$c$] at (3/2,3/4) {\tiny{$\bullet$}};
\node (f)[point,label=above:$f$] at (-1,-5/4) {\tiny{$\bullet$}};
\node (d)[point,label=right:$d$] at (1,-5/4) {\tiny{$\bullet$}};
\draw (a.center) -- (b.center);
\draw (a.center) -- (c.center);
\draw (a.center) -- (g.center);
\draw (a.center) -- (d.center);
\draw (a.center) -- (f.center);
\end{tikzpicture}
            &\quad \quad &
\begin{tikzpicture}
\tikzstyle{point}=[inner sep=0pt]
\node (a)[point,label=left:$a$] at (0,0) {\tiny{$\bullet$}};
\node (b)[point,label=below:$b$] at (-6/5,1) {\tiny{$\bullet$}};
\node (c)[point,label=right:$c$] at (0,3/2) {\tiny{$\bullet$}};
\node (d)[point,label=right:$d$] at (3/2,3/4) {\tiny{$\bullet$}};
\node (g)[point,label=left:$g$] at (-1,-5/4) {\tiny{$\bullet$}};
\node (f)[point,label=below:$f$] at (1/4,-6/4) {\tiny{$\bullet$}};
\node (e)[point,label=right:$e$] at (3/2,-1/2) {\tiny{$\bullet$}};
\draw (a.center) -- (b.center);
\draw (a.center) -- (c.center);
\draw (a.center) -- (g.center);
\draw (a.center) -- (e.center);
\draw (a.center) -- (d.center);
\draw (a.center) -- (f.center);
\draw (g.center) -- (f.center);
\draw (b.center) -- (c.center);
\draw (d.center) -- (e.center);
\end{tikzpicture}\\
          \end{tabular}
        \end{center}
\caption{A star graph and a star triangle}\label{f:star}
\end{figure}
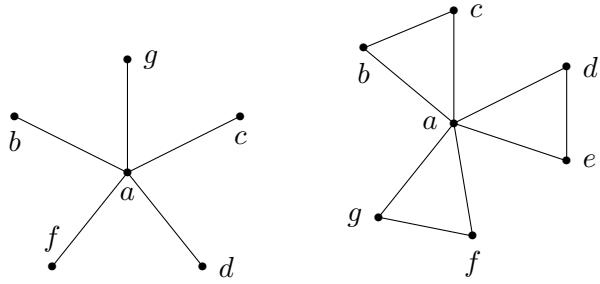

     \item a star triangle, a graph consisting of a set of triangles
       whose pairwise intersections are one single vertex as in
       \cref{f:star}; or

     \item a finite graph consisting of a connected bipartite graph
       with two vertex partition sets such that there is at least one
       leaf edge attached to each vertex in one vertex set and that
       there may be possibly some pendant triangles attached to the vertices
       in the other vertex set, as in \cref{CW}.  
       
\end{enumerate}

\begin{figure}[ht]
$$\begin{tikzpicture}
\tikzstyle{point}=[inner sep=0pt]
\node (a)[point,label=left:$a$] at (-2,2) {\tiny{$\bullet$}};
\node (b)[point,label=above left:$b$] at (0,2) {\tiny{$\bullet$}};
\node (c)[point,label=above left:$c$] at (2,2) {\tiny{$\bullet$}};
\node (d)[point,label=right:$d$] at (-3,0) {\tiny{$\bullet$}};
\node (e)[point,label=left:$e$] at (-1,0) {\tiny{$\bullet$}};
\node (f)[point,label= left:$f$] at (1,0) {\tiny{$\bullet$}};
\node (g)[point,label=right:$g$] at (3,0) {\tiny{$\bullet$}};
\node (h)[point,label=right:$h$] at (-3,3.5) {\tiny{$\bullet$}};
\node (i)[point,label=right:$i$] at (-1,3.5) {\tiny{$\bullet$}};
\node (j)[point,label=right:$j$] at (0,3.5) {\tiny{$\bullet$}};
\node (k)[point,label=right:$k$] at (1,3.5) {\tiny{$\bullet$}};
\node (l)[point,label=right:$l$] at (3,3.5) {\tiny{$\bullet$}};
\node (m)[point,label=below:$m$] at (-3.5,-1.5) {\tiny{$\bullet$}};
\node (n)[point,label=right:$n$] at (-2.5,-1) {\tiny{$\bullet$}};
\node (p)[point,label=above:$p$] at (-4,0.25) {\tiny{$\bullet$}};
\node (q)[point,label=below:$q$] at (-4.5,-1) {\tiny{$\bullet$}};
\node (r)[point,label=right:$r$] at (2.5,-0.5) {\tiny{$\bullet$}};
\node (s)[point,label=right:$s$] at (1.5,-1.5) {\tiny{$\bullet$}};
\node (t)[point,label=right:$t$] at (-2,4) {\tiny{$\bullet$}};
\draw (a.center) -- (h.center);
\draw (a.center) -- (d.center);
\draw (a.center) -- (f.center);
\draw (a.center) -- (i.center);
\draw (b.center) -- (j.center);
\draw (b.center) -- (d.center);
\draw (b.center) -- (e.center);
\draw (b.center) -- (f.center);
\draw (b.center) -- (g.center);
\draw (b.center) -- (k.center);
\draw (c.center) -- (l.center);
\draw (c.center) -- (f.center);
\draw (c.center) -- (g.center);
\draw (c.center) -- (e.center);
\draw (d.center) -- (m.center);
\draw (d.center) -- (n.center);
\draw (s.center) -- (r.center);
\draw (m.center) -- (n.center);
\draw (d.center) -- (q.center);
\draw (f.center) -- (s.center);
\draw (d.center) -- (p.center);
\draw (p.center) -- (q.center);
\draw (f.center) -- (r.center);
\draw (t.center) -- (a.center);
\end{tikzpicture}
$$
\caption{A Cameron--Walker Graph}
    \label{CW}
\end{figure}
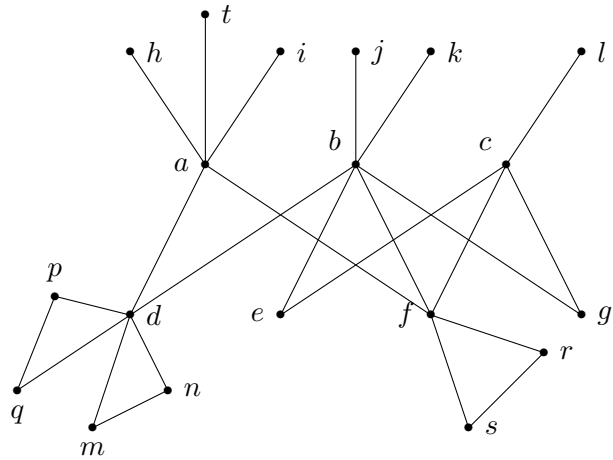

    Hibi et al \cite{Hi} defined any graph of type (3) to be  a {\bf Cameron-Walker graph}. i.e. a {\bf Cameron--Walker graph} is a finite connected simple graph $G$ such that $\m(G)=\im(G)$ and $G$ is neither a star graph nor a star triangle.
 
  If $R$ is a ring, the {\bf height of a prime ideal}
  $\mathfrak{p}$ is defined as the supremum of all $n$ where there is
  a chain $\mathfrak{p}_0 \subset ...\mathfrak{p}_{(n-1)} \subset
  \mathfrak{p}_n=\mathfrak{p}$ where all $\mathfrak{p}_i$ are distinct
  prime ideals.  The {\bf Krull dimension} or dimension of $R$ is
  defined as the supremum of all the heights of all its prime
  ideals. Let $M$ be a maximal ideal of $R$. A {\bf regular
    sequence} is a sequence $ a_{1}$, ..., $a_{n} \in M$ such that all
  $a_{i}$ are not zero-divisors in $R /( a_{1}, \ldots,
  a_{i-1})$. The {\bf depth} of a ring $R$ is the number of
  elements in some maximal regular sequence.  Suppose $G$ is a graph
  with vertices $\{x_{1}, \ldots, x_{n}\} $. Let $S=$ $K[x_{1},
    \ldots, x_{n}]$ denote the polynomial ring in $n$ variables
  over a field $K$.  The {\bf edge ideal of $G$}, denoted $I(G)$,
  is the ideal of $K[x_{1}, \ldots, x_{n}]$ with generators
  specified as follows: $x_{i} x_{j}$ is a generator of $I(G)$ if and
  only if $x_{i}$ is adjacent to $x_{j}$ in $G$.
 \begin{example} The edge ideal of the graph in \cref{G} is 
 $I(G)=(ab,bc,cd,de,ef,fa,bf,ce)$.
 \end{example}

\section{On the set {${\Grdd(n)}$}}

Let $\Graph(n)$ be the set of all finite simple graphs having $n$
vertices and let $\Grdd(n)$ denote the set of all possible pairs
$(\depth(G), \dim(G))$ arising from the elements $G$ in $\Graph(n)$. In
other words
$$ \Grdd(n)=\{(\depth(G),\dim(G) \mid  G\in \Graph(n)\}.$$

A natural question is: what are the possible  pairs of $(a,b) \in \Grdd(n) $ where $a,b\in \mathbb{N}$? It is not difficult  to compute the pairs $(a,b)$ for the simple graphs with small $n$.

\begin{example}\cite[Example 2.1]{h}
A connected graph $G$ with three vertices is either a path on three
vertices or a triangle. If $G$ is a path, we have the edge ideal
$I(G)= (x_1x_2,x_2x_3) \subset K[x_1,x_2,x_3]$ 
  {and} 
  $$
\depth(G)=1,\quad \dim(G)=2.$$ If $G$ is a triangle, then  
$I(G)=(x_1x_2,x_2x_3,x_2x_3)$  and $\depth(G)= \dim(G)=1$.
\end{example}

While finding $\Grdd(n)$ is easy for $n\leq3$, for larger $n$, determining
the elements of this set becomes more complicated. The authors in
\cite{h} give lattice boundaries for the points in this set.

\begin{theorem}\cite [Theorem 2.9]{h}
\label{3.1}
If $G$ is any simple graph on $n \geq 3$ vertices, then
$$ C^{-}(n) \subseteq \Grdd(n) \subseteq C^{+}(n)$$
where
$${C^{-}(n)= \{(1,n-1)\} \cup \left\{(a,b) \in \mathbb{N}^2 \mid 1 \leq a \leq b \leq n-2, \ a \leq \lf{\frac{n}{2}}\rf \right\} \subseteq \mathbb{N}^2}$$ 
and 
$$C^{+}(n) =  \left\{(a,b) \in \mathbb{N}^2 \mid  1\leq a \leq b \leq  n-1\right\} \subseteq \mathbb{N}^2.$$
\end{theorem}

In other words the set $\Grdd(n)$ is sandwiched by the convex lattice
polytopes $C^{-}(n)$ and $C^{+}(n)$. This calculates precise bounds
for the elements of $\Grdd(n)$ for $n \geq 3$. When restricted to the
class of Cameron--Walker graphs, they were able to give a full
characterization of the pairs $(a,b)$ that appear as $\depth(G)$ and
$\dim(G)$ of Cameron--Walker graphs.

\begin{theorem}\cite[Theorem 3.15]{h}
\label{3.15}
Let $\CWdd(n)$ denote the set of all possible pairs of depth and
dimension of edge ideals of Cameron--Walker Graphs with $n$ vertices.
Then for $n \geq 5$

\begin{equation}
\label{eq2}
\CW_{2,\dim}(n)= \begin{cases} 
  \{(2,n-2),(2,n-3)\},& n \mbox{ even}  \\
  \{(2,n-2),(2,n-3),\left(2,\frac{n-1}{2}\right)\}
      & n \mbox{ odd}
   \end{cases}
\end{equation}
and 
\begin{equation}\label{e:ABC}
\begin{aligned}
     \CWdd(n) &=\underbrace{\CW_{2,\dim (n)}}_{A} \cup
       \underbrace{\left\{(b,b) \in \mathbb{N}^2\ |\ \frac{n}{3} < b <
         \frac{n}{2}\right\}}_{B}\\&\cup \underbrace{\left\{(a,b) \in
         \mathbb{N}^2 \mid 3\leq a \leq
         \lf{\frac{n-1}{2}}\rf, \max
         \left\{a,\frac{n-a}{2}\right\}< b \leq n-a\right\}}_{C}.
\end{aligned}
\end{equation}
\end{theorem}
 
       We now determine the size of the set $\CWdd(n)$ based on the
       description of the elements.  This will allow one to find how
       many Cameron--Walker graphs there are of a fixed dimension and
       depth, over a given number of vertices. To do so, we determine
       the size of each of the sets $A$, $B$ and $C$ in \eqref{e:ABC}
       and their overlaps.

\begin{lemma}\label{Lemma3.4}
 For the sets, $A$, $B$ and $C$ are defined as in \eqref{e:ABC}, $A\cap
 C=B\cap C=\emptyset$, and
$$A\cap B=\begin{cases} 
    \{(2,2)\}& n=5 \\ \emptyset
      & \mbox{otherwise.}
      \end{cases} $$ 
 \end{lemma}
 
 \begin{proof}
If $(b,b)\in A\cap B$, then in \eqref{eq2} we must have $2=n-2$ or
$2=n-3$ or $2=\frac{n-1}{2}$. So $n=4$ or $n=5$. But
there is no CW graphs with $n=4$, so $n=5$, and in this case $(2,2)$
is the only point in $A\cap B$.

 If $(b,b) \in B\cap C$, then by the definition of $C$, 
$\max \left\{b,\frac{n-b}{2}\right\} \lneq   b$. Therefore, we have $b \lneq b$, which is a contradiction.

If $(a,b) \in A \cap C$,  then $a \geq 3$ and $a=2$, which is not possible. Thus, $A\cap C= \emptyset$.
 
 \end{proof}

 \begin{lemma}[{\bf The size of $A$}]
   If $A$ is as in \cref{3.15} and  $n\geq5$, then 
 \begin{equation}
\label{eq4}
     |A|=\begin{cases} 
     2 &  n \mbox{ is even or }  n=5 \\
     3 & \mbox{otherwise}.
      \end{cases}
      \end{equation}
\end{lemma}

      \begin{proof}
      By \eqref{eq2} in \cref{3.15}, it is easy to see that for each
      $n>5$ we have two elements in $\mbox{CW}_{\mbox{2,dim}}(n)$ if
      $n$ is even and we have three elements if $n$ is odd.  When
      $n=5$, we have the two elements: $(2,3)$ and $(2,2)$ in
      $\mbox{CW}_{\mbox{2,dim}}(n)$.  Our claim now follows.
       \end{proof}
           
\begin{lemma}[{\bf The size of $B$}]\label{Lemma3.6}
With $B$ is as in \cref{3.15}, and $n\geq 5$, let $n=6k+i$ where $k
\geq 0$ and $0 \leq i \leq 5$. Then
\begin{equation}\label{eq3}
    \left|B\right|=\begin{cases} 
   k-1  & i=0\\
   k   & 1\leq i \leq 4 \\
   k+1   & i=5.
    \end{cases}
   \end{equation}   
\end{lemma}

\begin{proof} 
  By definition
$$B
= \left \{ (b,b) \in \mathbb{N}^2 \mid \frac{n}{3} < b < \frac{n}{2}\right \}
= \left \{ (b,b) \in \mathbb{N}^2 \mid 2k+\frac{i}{3} < b < 3k+\frac{i}{2}\right \}.
  $$
The options for $b$ are then
$$  b \in \begin{cases} 
  \{2k+1,2k+2,\ldots,3k-1\} & i=0\\
  \{2k+1,2k+2,\ldots,3k  \} & i=1,2\\
  \{2k+2,2k+3,\ldots,3k+1 \} & i=3,4\\
  \{2k+2,2k+3,\ldots,3k+2 \} & i=5\\
    \end{cases}
$$
which proves our claim.
\end{proof}

\begin{lemma}[{\bf The size of $C$}]\label{Lemma3.7}
With $C$ is as in \cref{3.15}, $|C|=0$ when $n=5$.  When $n> 5$, let $n=6k+i$ where $k
> 0$ and $0 \leq i \leq 5$. Then 
 \begin{equation}
  \label{eq5}
   |C|=  \begin{cases} 
       6k^2-7k+1&  \ i =0\\
       6k^2-5k&  \ i =1 \\ 
    6k^2 -3k-1&  \ i =2\\
       6k^2-k-1&  \ i =3 \\ 
       6k^2+k-1&  \ i =4 \\
       6k^2 +3k-1&  \ i =5.
         \end{cases}
  \end{equation} 
  \end{lemma}

\begin{proof}
 Recall that $C=\left\{(a,b) \in \mathbb{N}^2 \ \Big|
   \ 3\leq a \leq \lf{\frac{n-1}{2}}\rf,\quad \max
   \left\{a,\frac{n-a}{2}\right\}< b \leq n-a\right\}.$
  
    We fix an integer $a$ with satisfying inequality  $ 3\leq a \leq \lf{\frac{n-1}{2}}\rf$. 
   Then we have two  possible scenarios
   $$ a \leq  \frac{n-a}{2} \qor a >  \frac{n-a}{2}.$$ 
     We consider each case separately.
 \begin{enumerate}
      \item When ${a \leq  \frac{n-a}{2}}$ (or equivalently $a \leq \frac{n}{3}$) we are looking for all the possible $b$'s satisfying the inequality $$\frac{n-a}{2}< b\leq n-a.$$
   \begin{enumerate}
      \item If $n-a$ is odd then the number of $b$'s satisfying the above inequality  is $$\frac{n-a-1}{2}+1=\frac{n-a+1}{2}= \lc \frac{n-a}{2}\rc.$$
    
      \item If $n-a$ is even then the number of $b$'s satisfying the above inequality is 
      $${\frac{n-a}{2}= \lc \frac{n-a}{2}\rc}.$$
  \end{enumerate}
  \item When ${a >  \frac{n-a}{2}}$ (or equivalently $a > \frac{n}{3}$), the number of $b$'s satisfying   $a < b \leq n-a$  is $n-2a$.  
  \end{enumerate} 
 Putting (1) and (2) together, and observing that since $n\geq 6$ we must have $\frac{n}{3} < \frac{n-1}{2}$, we conclude that  
  \begin{equation}\label{e:C}
  \left|C\right| = \sum_{a=3}^{\lf{\frac{n}{3}}\rf}
                   \lc \frac{n-a}{2}\rc  + 
                   \sum_{a={\lf{\frac{n}{3}}\rf}+1}^{\lf{\frac{n-1}{2}}\rf}(n-2a).
  \end{equation}
  Mathematica \cite{Mt} computes this sum  to be \eqref{eq5}.
\end{proof}

     \begin{theorem}[{\bf The size of $\CWdd(n)$}]\label{4.2}
 Let $n$ be an integer. If $n>5$ let $n=6k+i$ where $k
\geq 1$ and $0 \leq i \leq 5$. 
$$|\CWdd(n)|=
\begin{cases} 
0& n<5\\
2& n=5\\     
6k^2-6k+2& i=0  \\ 
6k^2-4k+3& i=1\\
 6k^2-2k+1& i=2\\ 
 6k^2+2& i=3\\
 6k^2+2k+1& i=4\\
6k^2+4k+3& i=5.
\end{cases}$$
   \end{theorem}

\begin{proof} 
There is no Cameron--Walker graphs with less than $5$ vertices, so
assume $n\geq 5$.  By \cref{3.15}, using elementary set theory and
\cref{Lemma3.4}, we
have $$\left|\CWdd(n)\right|=\begin{cases}
{|A|+|B|+|C|-|A\cap B|}& \mbox{if} \ n=5\\
 
{|A|+|B|+|C|}& \mbox{if} \ n>5\\

\end{cases}$$
   Hence,  using  \eqref{eq4}, \eqref{eq3}, and \eqref{eq5}; when $n=6k$ (or when $n>5$),
  $$\left|\CWdd(n)\right|=2+ (k-1) + (k-1)(6k-1)=6k^2-6k+2.$$
  Similarly, we can compute the other cases.
 Moreover, when $n=5$,  
$$\left|\CWdd(n)\right|=2+1+0 -1=2.$$
  
  \end{proof}

The following statement is a direct consequence of \cref{4.2}.

\begin{corollary}
 $$\lim_{k \rightarrow \infty}\frac{\left|\CWdd(n)\right|}{n^2}=\frac{1}{6}.$$
\end{corollary}

Let \begin{eqnarray*}
  \rd(n) &=&\{(\reg(G),\degh(G)) \mid  G\in \CW(n)\}\\
  \grd(n) &=&\{(\reg(G),\degh(G)) \mid G\in \Graph(n)\}
\end{eqnarray*}
where $\reg(G)$ and $\degh(G)$ are the regularity and
   the degree of the $h-$polynomial of the edge ideals of $G$,
   respectively. The authors in \cite{h} 
   gave a precise description
    for the size of the set $\rd(n)$ when $n\geq 5$.
    
It would be nice to
compare the number of integer points in $\CWdd(n)$ to the number of integer points in $\Grdd(n)$ similar to the comparison described  between the $\rd(n)$ and $\grd(n)$ in \cite{h}. 
Then one might be able to find the percentage of lattice points recognized by the Cameron--Walker graphs.

Thus it would be interesting if we can answer the following question.
\begin{question}
What is the value of $${\lim_{n \rightarrow \infty} \frac{\left|\CWdd(n)\right|}{\left|\Grdd(n)\right|}}?$$
\end{question}

In  \cite{h}, the authors asked a similar question for the $\rd(n)$. 
Since we can only bound $\Grdd(n)$, it is not clear enough whether the this limit exists or not. If we want to show that this limit exists, it is enough to show that $${\frac{|\CWdd(n)|}{|\Grdd(n)|}}$$ is  increasing for all $n$ as its value is always $\leq 1.$ The authors in \cite{h} observed,  we can use the monotone convergence theorem to show that the limit exists. Computational evidence shows that this fraction is indeed increasing and the limit exists.

\begin{theorem}
Suppose  ${\lim_{n \rightarrow \infty} \frac{\left|\CWdd(n)\right|}{\left|\Grdd(n)\right|}}$ exists. Then
$$\frac{1}{3} \leq {\lim_{n \rightarrow \infty} \frac{|\CWdd(n)|}{|\Grdd(n)|}} \leq \frac{4}{9}.$$
\end{theorem}

\begin{proof} 
Setting $n=6k+i$ for $k>0$ and $0 \leq i \leq 5$, we have  $(n-3)^2=36k^2+12k(i-3)+(i-3)^2$. On the other hand, for each value of $i$, by \cref{4.2} there is an integer $\epsilon_i \in \{0,1,2,3\}$ such that  
$$\left|\CWdd(n)\right|=  6k^2+2(i-3)k + \epsilon_i.$$
 Putting all this together, we have 
 $$\left|\CWdd(n)\right|=\frac{1}{6}(n-3)^2+ \epsilon_i-\frac{1}{6}(i-3)^2.$$ 
 Using \cref{4.2} again by setting $i=0,\ldots,5$, we observe that 
 $$\epsilon_i-\frac{1}{6}(i-3)^2 \in 
   \left \{ 2-\frac{9}{6}, 3-\frac{4}{6}, 1-\frac{1}{6}, 2, 1-\frac{1}{6}, 3-\frac{4}{6} \right \}
 = \left \{ \frac{1}{2}, \frac{7}{3}, \frac{5}{6}, 2 \right \}
 .$$
 We have therefore shown that if $n>5$,
 for some $u$ such that ${\frac{1}{2}\leq u \leq \frac{7}{3}}$ we have  
\begin{equation}
\label{6}
 \frac{1}{6}(n-3)^2+\frac{1}{2} \leq \left|\CWdd(n)\right| \leq \frac{1}{6}(n-3)^2+ \frac{7}{3}   
\end{equation}

Now using \cref{3.1}, we can get an upper and a lower bound for the ${|\Grdd(n)|}$.
i.e. $${\left|C^{-}(n)\right|\leq\left|\Grdd(n)\right|\leq\left|C^{+}(n)\right|.}$$
Now recall from \cref{3.1}, that 
$${C^{-}(n) = \left\{(1,n-1)\right\} \cup \left\{\left(a,b\right) \in \mathbb{N}^2 \mid a \leq b, 1\leq a \leq \lf{\frac{n}{2}}\rf, 1 \leq b\leq n-2\right\} \subseteq \mathbb{N}^2}.$$
Let $${\beta=\left\{(a,b) \in \mathbb{N}^2 \ |\ 1\leq a \leq \lf{\frac{n}{2}}\rf, \ a  \leq b\leq n-2\right\} \subseteq \mathbb{N}^2}.$$
As ${\lf{\frac{n}{2}}\rf \leq n-2}$ for all $n\geq 4$, the number of possibilities for $b$ in the set $\beta$ can be found by the inequality $ a  \leq b\leq n-2$. Therefore,
$$\left|\beta\right|=\sum_{a=1}^{\lf{\frac{n}{2}}\rf} (n-a-1).$$

Now, recall that $n = 6k + i$.
So,  Mathematica \cite{Mt} computes $|\beta|$  to be, 
    $$
    |\beta|=\begin{cases} 
    \frac{9}{2}k(3k-1)& i=0 \\ 
    \frac{3}{2}k(9k-1)& i=1\\
    \frac{9}{2}k(3k+1)& i=2\\ 
    \frac{1}{2}(3k+1)(9k+2)& i=3\\ 
    \frac{3}{2}(3k+2)(3k+1)& i=4 \\ 
    \frac{1}{2}(3k+2)(9k+5)& i=5   
    \end{cases}
    \qand 
    |C^{-}(n)|=
   \begin{cases} 
    \frac{1}{2}(27k^2-9k)+1& i=0 \\ 
    \frac{1}{2}(27k^2-3k)+1& i=1\\
    \frac{1}{2}(27k^2+9k)+1& i=2\\ 
    \frac{1}{2}(27k^2+15k+2)+1& i=3\\ 
    \frac{1}{2}(27k^2+27k+6)+1& i=4 \\   
    \frac{1}{2}(27k^2+33k+10)+1& i=5. 
    \end{cases}$$

   Since $|C^{-}(n)|$ is the lower bound for ${|\Grdd(n)|,}$ combining this with the upper bound for ${|\CWdd(n)|}$ in \eqref{6} gives
    $$\lim_{k \rightarrow \infty} \frac{|\CWdd(n)|}{|\Grdd(n)|}\leq 
   \lim_{k \rightarrow \infty} \frac{6k^2}{\frac{1}{2} 27k^2}=\frac{4}{9}.
   $$

    Similarly by \cref{3.1} an upper bound for $|\Grdd(n)|$ is
       \begin{align*}
      |C^{+}(n)|=& \sum_{a=1}^{n-1} (n-a)=n(n-1)- \frac{(n-1)n}{2}=\frac{n}{2}(n-1).
      \end{align*}

      Now, for  $ n > 5$  let $n = 6k + i.$
      So, we have
\begin{align*}
      |C^{+}(n)|=& \sum_{a=1}^{6k+i-1} (6k+i-a)=\frac{6k+i}{2}(6k+i-1).
      \end{align*}

      Thus combining the lower bound of $|\Grdd(n)|$ with $|C^{+}(n)|$,
      we get the lower bound for 
      $$\lim_{k\rightarrow \infty}\frac{|\CWdd(n)|}{|\Grdd(n)|}\geq \lim_{k\rightarrow \infty}\frac{6k^2}{\frac{1}{2}36k^2}=\frac{1}{3}$$  
      \end{proof}

\section {the Size of $\Ra(n)$}

In \cite{h1} the authors determined the size of the set $\RD(n)$ , which is the set of all possible pairs of $(\reg(G), \degh(G)) $ arising from all the Cameron--Walker graphs with $n$ vertices. Also, we already determined the size of the set $\CWdd(n)$. Hibi et al  \cite{h} gave a precise description for the set $\Ra(n)$. Based on their characterization we determine the size of the set $\Ra(n)$ through this section.

\begin{theorem}\cite[Theorem 4.4]{h}
\label{4.4*}
Let $n\geq 5$ be an integer. Then 
\begin{align*}
    \Ra(n)=&\underbrace{\Raa(n)} _{A} \\&\cup  \underbrace{\left\{(a,d,d,d) \in \mathbb{N}^4\ \biggr\rvert \ 3\leq a\leq d\leq\lf\frac{n-1}{2}\rf,\  n<a+2d\right\} }_{B}  \\&\cup  \underbrace{\{(a,a,d,d) \in \mathbb{N}^4\ \vert \ 3\leq a< d\leq n-a, \ n\leq 2a+d-1\}}_{C}\\ &\cup  \underbrace{\{(a,r,d,d) \in \mathbb{N}^4\ \vert \ 3\leq a<r< d< n-r, \ n+2\leq a+r+d\}}_{D}\end{align*}
    where 
 \begin{equation}
    \label{eq7}
       A=
        \begin{cases}
        \{(2,2,n-2,n-2),(2,2,n-3,n-3)\}& n \quad \mbox{even}\\
        \left \{(2,2,n-2,n-2),(2,2,n-3,n-3),\left(2,\frac{n-1}{2},\frac{n-1}{2},\frac{n-1}{2}\right)\right\} & n \quad \mbox{odd}.
    \end{cases}
    \end{equation}
    \end{theorem}

\begin{lemma}
\label{Lemma4.2}
 The sets $A, B, C$ and $D$  defined in \cref{4.4*} are pairwise disjoint.
\end{lemma}   

 \begin{proof}
     Since $a\geq 3$ in $B,C$ and $D$, we immediately see that $A\cap B=A\cap C=A\cap D=\emptyset$. In $C$, we have $\depth(G)=\reg(G)$. So consider the case $\depth(G)=\reg(G)$  in $B$. Then we have the tuple $(a,a,a,a)$. However, $(a,a,a,a)\notin C$ as we have the additional condition $a<d$. Therefore, $B\cap C=\emptyset$. 

As $\reg(G)<\dim(G)$ in $D$,  $B\cap D= \emptyset$ and using the same reasoning $\depth(G)<\reg(G)$  in $D$, finally we get  $C\cap D= \emptyset$.
 \end{proof}

 \begin{lemma}[{\bf The size of $A$}]\label{l:A}
   If $A$ is as in the statement of \cref{4.4*} and  $n\geq5$, then 
$$|A|=
        \begin{cases}
        2 & n \quad \mbox{even} \qor n=5\\
         3& n \quad  \mbox{odd}.
    \end{cases}$$
\end{lemma}

      \begin{proof}
  By   \eqref{eq7}, it is easy to see that for each $n>5$ we have two elements in $\Raa(n)$ if $n$ is even and  we have three elements if $n$ is odd.
 When $n=5$, we have the two elements: $(2,2,3,3)$ and $(2,2,2,2)$ in $\Raa(n)$.  This settles our claim.
        \end{proof}
            
\begin{lemma}[{\bf The size of $B$}]\label{l:B}
With $B$ is as in the statement of \cref{4.4*} and $n\geq 5$, let $n=6k+i$ where $k
\geq 0$ and $0 \leq i \leq 5$. Then 
     \begin{align*}
       |B|=\begin{cases}
          \frac{1}{2} (3k^2-3k)& i=0\\
   \frac{1}{2}(3k^2+k-2) & i=1\\
   \frac{1}{2}(3k^2-k)& i=2\\
   \frac{1}{2}(3k^2+3k-2)& i=3\\
   \frac{1}{2}(3k^2+k)& i=4\\
   \frac{1}{2}(3k^2+5k)& i=5.
       \end{cases}
   \end{align*}
\end{lemma}

\begin{proof} 
 Recall that $B=\left\{(a,d,d,d) \in \mathbb{N}^4\ \biggr\rvert \ 3\leq a\leq d\leq\lf\frac{n-1}{2}\rf,\  n<a+2d\right\}$. Let $a$ and $d$ be integers such that $$3\leq a  \qand \max \{ a-1, \frac{n-a}{2}\} < d \leq \lf\frac{n-1}{2}\rf .$$ 
Then we have two scenarios:

\begin{enumerate}
    \item when $a-1 < \frac{n-a}{2}$, or equivalently $a\leq \frac{n}{3}$, we have $\frac{n-a}{2} < d \leq \lf\frac{n-1}{2}\rf$.  Thus, there are $\lf\frac{n-1}{2}\rf-\lf\frac{n-a}{2}\rf$ possibilities for $d$.
    
    \item When $a-1 \geq \frac{n-a}{2}$, or equivalently $a> \frac{n}{3}$, the number of $d$'s satisfying ${(a,d,d,d)\in B}$ is given by ${a \leq d \leq \lf\frac{n-1}{2}\rf}$. In other words  the number of possible $d$'s is $\lf\frac{n-1}{2}\rf -a+1$. 
\end{enumerate}

$$|B|= \sum_{a=3}^{\lf \frac{n}{3}\rf}  \left(\lf \frac{n-1}{2} \rf-\lf\frac{n-a}{2}\rf\right)+
         \sum_{a=\lf \frac{n}{3}\rf+1}^{\lf \frac{n-1}{2} \rf}  \left(\lf \frac{n-1}{2} \rf-(a-1)\right) $$
 We can then simplify this expression using  Mathematica \cite{Mt}.
\end{proof}

\begin{lemma}[{\bf The size of $C$}]\label{l:C}
With $C$ is as in the statement of \cref{4.4*}, and $n\geq 5$, let $n=6k+i$ where $k
\geq 0$ and $0 \leq i \leq 5$. Then 
 \begin{equation}
  |C|=\begin{cases}
            3k^2-3 & i =0\\
            3k^2+k-3 & i =1\\
            3k^2+2k-2 & i =\mbox{2}\\ 
            3k^2+3k-2 & i =3\\ 
            3k^2+4k-2 & i =4\\ 
            3k^2+5k& i =5.
           \end{cases}
    \end{equation} 
  \end{lemma}

\begin{proof}
Recall that $C=\left\{(a,a,d,d) \in \mathbb{N}^4\ \vert \ 3\leq a< d\leq n-a, \ n\leq 2a+d-1\right\}$. Let $a$ and $d$ be integers such that $$a \geq 3 \qand \max\{a, n-2a+1\} < d\leq n-a.$$
So for a fixed $a\geq 3$ we consider two cases.
\begin{enumerate}
    \item When $a\leq n-2a+1$, or equivalently when $a\leq \lf\frac{n+1}{3}\rf$ we must have ${n-2a+1\leq d\leq n-a}$. So, the number of possibilities for $d$ is $a$.
    
    \item When $a> n-2a+1$, or equivalently when $a> \lf\frac{n+1}{3}\rf$, we have $a< d \leq n-a.$ So, the possible $d$ in this case is $n-2a$.
\end{enumerate}

$$|C|=\sum_{a=3}^{\lf \frac{n+1}{3} \rf}a +  \sum_{a=\lf \frac{n+1}{3} \rf+1}^{\lf \frac{n-1}{2} \rf}n-2a$$  
 Mathematica \cite{Mt} will do the rest of the calculations resulting.
\end{proof}

\begin{lemma}[{\bf The size of $D$}]\label{l:D}
If $D$ is as in the statement of \cref{4.4*} and $n=5,$ then $|D|=0.$ When $n> 5$, let $n=6k+i$ where $k
> 0$ and $0 \leq i \leq 5$. Then  
\begin{align*}
    |D|&=\begin{cases}
        \frac{1}{8} (18k^3-45k^2+34k-8) & i=0 \qand k \ \text{even}\\
          \frac{1}{8} (18k^3-45k^2+34k-7) & i=0 \qand k \ \text{odd}\\
            \frac{1}{8} (18k^3-27k^2-14k+24) & i=1 \qand k \ \text{even}\\
            \frac{1}{8} (18k^3-27k^2-14k+23)& i=1 \qand  k \ \text{odd}\\   
            \frac{1}{8} (18k^3-27k^2+10k) & i=2 \qand k \ \text{even}\\    
            \frac{1}{8} (18k^3-27k^2+10k-1) & i=2 \qand k \ \text{odd}\\
          \frac{1}{8} (18k^3-9k^2-26k+16) & i=3 \qand k \ \text{even}\\
          \frac{1}{8} (18k^3-9k^2-26k+17) & i=3 \qand k \ \text{odd}\\
          \frac{1}{8} (18k^3-9k^2-2k)& i=4 \qand k \ \text{even}\\
         \frac{1}{8} (18k^3-9k^2-2k+1) & i=4 \qand k \ \text{odd}\\
         \frac{1}{8} (18k^3+9k^2-26k+8)& i=5 \qand k \ \text{even}\\
         \frac{1}{8} (18k^3+9k^2-26k+7)& i=5 \qand k \ \text{odd}\\  
         \end{cases}
\end{align*}

   \end{lemma}       

\begin{proof}
Recall that $D= \{(a,r,d,d) \in \mathbb{N}^4\ \vert \ 3\leq a<r< d< n-r, \ n+2\leq a+r+d\}$. 
Now $r<d<n-r$ implies that $r \leq d-1 \leq n-r-2$, and so we must have $r \leq \frac{n-2}{2}=\frac{n}{2}-1$. So 
we fix integers $a$ and $r$ such that
$$3 \leq a \leq \lf\frac{n}{2}\rf-2 \qand a+1 \leq r \leq \lf\frac{n}{2}\rf -1 .$$
Now using both the first and  second inequalities in $D$, we can deduce the following conditions on $d$
$$\max\{r+1, n-a-r+2\} \leq d  <n-r,$$
which leads us to consider two scenarios.
\begin{enumerate}
    \item When $r+1\leq n-a-r+2$, or in other words, ${r< \frac{(n-a)}{2}}+1$,  then $d$ satisfies the inequality  $n-r-a+2\leq d < n-r$. Therefore, we have  $(a-2)$ such integers $d$.

  \item When $r+1 > n-a-r+2$, or equivalently ${r\geq \frac{(n-a)}{2}+1}$, we can find the number of $d$'s by the inequality $r<d<n-r$. Thus, the number of such $d$ is $(n-2r-1)$. 
\end{enumerate}

So the total number of $d$'s for a fixed $a$ and $r$ is  
\begin{align*}
&\sum_{r=a+1}^{\lf\frac{n-a}{2}\rf} (a-2) 
+ \sum_{r=\lf\frac{n-a}{2}\rf+1}^{\lf\frac{n}{2}\rf -1}(n-2r-1)\\
=& (a-2) \left (\lf\frac{n-a}{2}\rf-a \right )
+  \left ( \lf\frac{n}{2}\rf -1 - \lf\frac{n-a}{2}\rf \right ) 
\left( n-1-\left (  \lf\frac{n}{2}\rf  +\lf\frac{n-a}{2}\rf \right ) \right ).
\end{align*}
Now summing over $a$ we have
$$ |D|=\sum_{a=3}^{\lf\frac{n}{2}\rf-2}\left( 
(a-2) \left (\lf\frac{n-a}{2}\rf-a \right )
+  \left ( \lf\frac{n}{2}\rf -1 - \lf\frac{n-a}{2}\rf \right ) 
\left( n-1-  \lf\frac{n}{2}\rf  -\lf\frac{n-a}{2}\rf  \right )
 \right).$$

The rest of the computation is carried out  by  Mathematica \cite{Mt}.
\end{proof}

\begin{theorem}
With $\Ra(n)$ is as in the statement of \cref{4.4*}, $|\Ra(n)|=2$ when $n=5$.  When $n> 5$, let $n=6k+i$ where $k
> 0$ and $0 \leq i \leq 5$. Then 

\begin{align*}
       | \Ra(n)|&=\begin{cases}
        \begin{cases}
         \frac{1}{8} (18k^3-9k^2+22k-16)& k \ \text{is even}\\
           \frac{1}{8} (18k^3-9k^2+22k-15) & k \ \text{is odd}
        \end{cases}& i=0\\
        \begin{cases}
          \frac{1}{8} (18k^3+9k^2-2k+16)  & k \ \text{is even}\\
            \frac{1}{8} (18k^3+9k^2-2k+15) & k \ \text{is odd}
      \end{cases}& i=1\\
        \begin{cases}
        \frac{1}{8} (18k^3+9k^2+22k)& k \ \text{is even}\\
            \frac{1}{8}  (18k^3+9k^2+22k-1)& k \ \text{is odd}
            \end{cases}& i=2\\
        \begin{cases}
           \frac{1}{8} (18k^3+27k^2+10k+16)& k \ \text{is even}\\
          \frac{1}{8} (18k^3+27k^2+10k+17) & k \ \text{is odd}
        \end{cases}& i=3\\
        \begin{cases}
          \frac{1}{8} (18k^3+27k^2+34k)& k \ \text{is even}\\
           \frac{1}{8} (18k^3+27k^2+34k+1)  & k \ \text{is odd}
        \end{cases}& i=4\\
        \begin{cases}
        \frac{1}{8} (18k^3+27k^2+34k+32) & k \ \text{is even}\\
         \frac{1}{8} (18k^3+27k^2+34k+31)& k \ \text{is odd}
          \end{cases}& i=5\\
    \end{cases}
   \end{align*}

\end{theorem}

 \begin{proof}
 Using Lemmas \ref{Lemma4.2},\ref{l:A},\ref{l:B},\ref{l:C},and \ref{l:D}.
\end{proof}

It is natural to wonder how many Cameron--Walker graphs achieve a given tuple of integers on their sequence of homological invariants. This question was raised by Adam Van Tuyl \cite{VT}, and we are able to partially answer it below.

\begin{proposition}
 \label{propo 4.5}
    Let $2\leq a \leq b\leq n-2$, and let $G$ be a Cameron--Walker graph. 
    \begin{enumerate}
        \item $\depth (G)=\dim  (G)=(b,b)$ if and only if $G$ is a Cohen Macaulay Cameron--Walker graph with  exactly $b$ vertices in its bipartite part.
        \item Suppose $(\depth (G),\dim (G))=(2,b)$. 
        \begin{enumerate}
            \item    $\dim(G)=n-2$ if and only if $m=2$ and $t_j=0$ for all $1\leq j \leq p$.
            \item $\dim(G)=n-3$ if and only if $m=p=1$ and $t_1=1$.
        \item     ${\dim(G)=\frac{n-1}{2}}$ and $n$ is odd if and only if $m=p=1$, and $s_1=1$, and $t_1\geq 2$.
        \end{enumerate}
    \end{enumerate}
\end{proposition}
 
 \begin{proof}

\begin{enumerate}
    \item $(\Rightarrow)$ is clear by the Cohen--Macaulay property.
    
     $(\Leftarrow)$ By \cite [Theorem 1.3]{Hi}, if $G$ is a Cohen--Macaulay Cameron--Walker graph then $s_i=t_j=1$ for each $i=1,2,\ldots,m$ and $j=1,2,\ldots,p$.

So,
$(\depth(G),\dim(G))=(b,b) \Leftrightarrow$
$G$ is of the form in the \cref{CCW} where $G$ is a Cohen Macaulay Cameron--Walker graph with exactly $b$ vertices in its bipartite part.

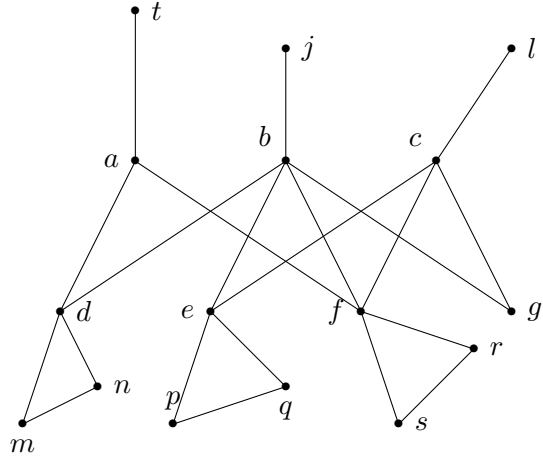
\begin{figure}[ht]
$$\begin{tikzpicture}
\tikzstyle{point}=[inner sep=0pt]
\node (a)[point,label=left:$a$] at (-2,2) {\tiny{$\bullet$}};
\node (b)[point,label=above left:$b$] at (0,2) {\tiny{$\bullet$}};
\node (c)[point,label=above left:$c$] at (2,2) {\tiny{$\bullet$}};
\node (d)[point,label=right:$d$] at (-3,0) {\tiny{$\bullet$}};
\node (e)[point,label=left:$e$] at (-1,0) {\tiny{$\bullet$}};
\node (f)[point,label= left:$f$] at (1,0) {\tiny{$\bullet$}};
\node (g)[point,label=right:$g$] at (3,0) {\tiny{$\bullet$}};
\node (j)[point,label=right:$j$] at (0,3.5) {\tiny{$\bullet$}};
\node (l)[point,label=right:$l$] at (3,3.5) {\tiny{$\bullet$}};
\node (m)[point,label=below:$m$] at (-3.5,-1.5) {\tiny{$\bullet$}};
\node (n)[point,label=right:$n$] at (-2.5,-1) {\tiny{$\bullet$}};
\node (p)[point,label=above:$p$] at (-1.5,-1.5) {\tiny{$\bullet$}};
\node (q)[point,label=below:$q$] at (0,-1) {\tiny{$\bullet$}};
\node (r)[point,label=right:$r$] at (2.5,-0.5) {\tiny{$\bullet$}};
\node (s)[point,label=right:$s$] at (1.5,-1.5) {\tiny{$\bullet$}};
\node (t)[point,label=right:$t$] at (-2,4) {\tiny{$\bullet$}};

\draw (a.center) -- (d.center);
\draw (a.center) -- (f.center);
\draw (p.center) -- (q.center);
\draw (b.center) -- (j.center);
\draw (b.center) -- (d.center);
\draw (b.center) -- (e.center);
\draw (b.center) -- (f.center);
\draw (b.center) -- (g.center);
\draw (c.center) -- (l.center);
\draw (c.center) -- (f.center);
\draw (c.center) -- (g.center);
\draw (c.center) -- (e.center);
\draw (d.center) -- (m.center);
\draw (d.center) -- (n.center);
\draw (s.center) -- (r.center);
\draw (m.center) -- (n.center);
\draw (e.center) -- (q.center);
\draw (f.center) -- (s.center);
\draw (e.center) -- (p.center);
\draw (f.center) -- (r.center);
\draw (t.center) -- (a.center);
\end{tikzpicture}
$$
\caption{A Cohen Macaulay Cameron--Walker Graph as in  \cref{propo 4.5}.}
    \label{CCW}
 \end{figure}

 \item 
  In this case, $\depth (G) =2$ and $\dim (G)=b$. By \cite [Lemma 3.11]{h} we have three types of Cameron--Walker graphs with the property $\depth(G)=2$.

 \begin{enumerate}
     \item 
     $(\Rightarrow)$ is clear by \cite [Proposition 3.13]{h}.
     
$(\Leftarrow)$     Suppose
 $m=2$ and $t_j=0$ for all $1\leq j\leq p$. 
 
By \cite [Lemma 3.12]{h},    ${\dim (G)=|V(G)|-2.}$
 Then we have $b=n-2$. i.e. if ${\dim (G)=n-2}$ then Cameron--Walker graph is of the form $m=2, \ t_j=0$ for all $j$.

 \item $(\Rightarrow)$ is clear by \cite [Proposition 3.13]{h}.
     
$(\Leftarrow)$ Suppose $m=p=1$, $t_1=1$.
 Then by \cite [Lemma 3.12]{h},    ${\dim (G)=|V(G)|-3.}$
i.e. $b=n-3$. 
Therefore, if  dim $G=n-3$ then Cameron--Walker graph is of the form $m=p=1, \ t_1=1$. 

 \item $(\Rightarrow)$ is clear by \cite [Proposition 3.13]{h}.
     
$(\Leftarrow)$ Suppose
  $m=p=1$, $s_1=1$, $t_1\geq2$.
  
  Then by \cite [Lemma 3.12]{h},    ${\dim (G)=\frac{|V(G)|-1}{2}.}$
i.e. ${b=\frac{n-1}{2}}$. 
  
 Thus, dim ${(G)=\frac{n-1}{2}}$ and $n$ is odd.
 \end{enumerate}
 \end{enumerate}
 \end{proof}

 \section{Acknowledgment}

 We would like to express our very great appreciation to Dr. Adam Van Tuyl for his valuable suggestions to improve this research work.

\end{document}